\documentclass[12pt, leqno]{amsart}
\usepackage{amsmath}
\usepackage{amssymb}
\usepackage{amsthm}
\usepackage{enumerate}
\usepackage[mathscr]{eucal}
\usepackage{xcolor}
\theoremstyle{plain}
\usepackage{tikz}
\usepackage{soul}
\usepackage[normalem]{ulem}
\usepackage[normalem]{ulem}
\newtheorem{theorem}{Theorem}[section]
\newtheorem{lemma}[theorem]{Lemma}
\newtheorem{prop}[theorem]{Proposition}
\newtheorem{cor}[theorem]{Corollary}

\theoremstyle{definition}
\newtheorem{definition}[theorem]{Definition}
\newtheorem{remark}[theorem]{Remark}

\newtheorem{example}[theorem]{Example}

\theoremstyle{remark}

\newcommand{\Ext}{{\rm Ext}\,}

\newcommand{\conv}{{\rm conv}\,}

\newcommand{\Li}{\mathcal{L}}
\newcommand{\K}{\mathcal{K}}

\numberwithin{equation}{section}

\begin{document}

		\title[$\rho$-symmetricity in the space of bounded linear operators]{On symmetricity of the norm derivatives orthogonality in operator spaces}
\author[Ghosh, Paul and  Sain]{Souvik Ghosh, Kallol Paul and Debmalya Sain}

\newcommand{\acr}{\newline\indent}

\address[Ghosh]{Department of Mathematics\\ Jadavpur University\\ Kolkata 700032\\ West Bengal\\ INDIA}
\email{sghosh0019@gmail.com}

\address[Paul]{Vice-Chancellor\\ University of Kalyani\\ West Bengal \\ \& Department of Mathematics\\ Jadavpur University (on lien)\\ Kolkata 700032\\ West Bengal\\ INDIA}
\email{kalloldada@gmail.com}

\address[Sain]{Department of Mathematics\\ Indian Institute of Information Technology, Raichur\\ Karnataka 584135 \\INDIA}
\email{saindebmalya@gmail.com}

	\begin{abstract}
	We investigate $\rho$-orthogonality and its local symmetry in the space of bounded linear operators. A characterization of Hilbert space operators with symmetric numerical range is established in terms of $\rho$-orthogonality. Further, we provide  characterizations of $\rho$-left and $\rho$-right symmetric operators on finite-dimensional Hilbert spaces. In the two-dimensional real case, we show that the only nonzero $\rho$-left (or $\rho$-right) symmetric operators are scalar multiples of orthogonal matrices. However, in any finite-dimensional Hilbert space of dimension greater than two, an operator is $\rho$-left (or $\rho$-right) symmetric if and only if it is the zero operator. For infinite-dimensional spaces, we show that within a large class of operators, the zero operator remains the only example of $\rho$-left and $\rho$-right symmetric operators.

	\end{abstract}
	
	\thanks{Souvik Ghosh would like to thank  CSIR, Govt. of India, for the financial support in the form of Senior Research Fellowship under the mentorship of Prof. Kallol Paul.} 
	
	\maketitle
	
\section{Introduction.} 
Orthogonality plays a pivotal role in the geometric study of Banach spaces. Among the various properties of orthogonality, the asymmetric nature of some orthogonalities, particularly, Birkhoff-James orthogonality has attracted significant interest. Notably, it is well known that the symmetry of Birkhoff-James orthogonality serves to characterize inner product spaces among all Banach spaces \cite[Th. 1]{J} when the dimension of the space is more than two. A localized version of this symmetry, referred to as left and right symmetricity, was first introduced and investigated in \cite{sain}. Since then, several researchers have contributed to the deeper understanding of symmetricity in the context of Birkhoff-James orthogonality. For further insights, we refer readers to \cite{CSS,GMPS, KST, SRBB, GSP} and the comprehensive monograph \cite{MPSbook}. In a similar spirit, the concept of symmetricity in a Banach space has also been explored for another important orthogonality type induced by the norm derivatives, known as $\rho$-orthogonality. The notions of $\rho$-left and $\rho$-right symmetricity were introduced and studied in \cite{GPS}. Motivated by these developments, the present article is devoted to investigate the symmetric properties of bounded linear operators on Hilbert spaces with respect to $\rho$-orthogonality.

Letters $\mathbb{X}, \mathbb{Y}$ denote  complex Banach spaces, whereas $\mathbb{H}$ is denoted as complex Hilbert space unless otherwise stated. The real and imaginary part of a complex number $z$ is denoted by $\Re (z)$ and $\Im(z)$, respectively. $\mathbb{X}^*$ is denoted as the dual of $\mathbb{X}.$ The symbol `$\dim \mathbb{X}$' is used to denote the dimension of the space $\mathbb{X}.$  $B_\mathbb{X}$ and $S_\mathbb{X}$ denote the unit ball and the unit sphere of $\mathbb{X},$ respectively. For a nonzero element $x\in \mathbb{X},$ the set of all supporting functionals of $x$ is denoted by $J(x)$ and defined as $J(x)=\{x^* \in S_{\mathbb{X}^*}: x^*(x)=\|x\|\}.$ An element $x \in S_{\mathbb X}$ is said to be  smooth  if $J(x)$ is singleton. If every element of $S_\mathbb{X}$ is smooth then $\mathbb{X}$ is  said to be a smooth Banach space. The set of all extreme points of the unit ball of a Banach space $\mathbb{X}$ is denoted by $\Ext B_\mathbb{X}.$ Given $x, y\in \mathbb{X}$, $x$ is said to be Birkhoff-James orthogonal \cite{B, J1} to $y$ if $\|x+\lambda y\|\geq \|x\|$ for all scalar $\lambda.$ Given a subset $M\subset \mathbb{X},$ $M^{\perp}=\{y\in \mathbb{X}: x\perp_B y, \forall x\in M\}.$ $\Li(\mathbb{X}, \mathbb{Y})(\Li(\mathbb{H}))$ denotes the space of all bounded linear operators defined between $\mathbb{X}$ and $\mathbb{Y}$ (on $\mathbb{H}$). $\K(\mathbb{X}, \mathbb{Y})$ denotes the space of all compact operators between $\mathbb{X}$ and $\mathbb{Y}.$ Given any $T\in \Li(\mathbb{X}, \mathbb{Y})$, the norm attainment set is denoted by $M_T$ and is defined by $M_T=\{x\in S_\mathbb{X}: \|Tx\|=\|T\|\}.$ Let $T\in \Li(\mathbb{X})$. The numerical range of $T$ is defined as $$W(T):=\{x^*(Tx): x\in S_\mathbb{X}, x^*\in S_{\mathbb{X}^*}, x^*(x)=1\}.$$ Numerical radius of $T,$  denoted by $w(T),$ is defined as $w(T)=\sup \{|\lambda|: \lambda \in W(T)\}.$ The maximal numerical range of $T\in \Li(\mathbb{H})$ is defined as $W_0(T):=\{\lim \langle Tx_n, x_n\rangle: x_n\in S_\mathbb{X}, \|Tx_n\|\to \|T\| \}.$ If $u\in \mathbb{X}$ is nonzero, the abstract numerical range of $z\in \mathbb{X}$ with respect to $(\mathbb{X}, u)$ is the compact convex subset of $\mathbb{C}$ given by $V(\mathbb{X}, u, z):=\{\phi(z): \phi\in S_{\mathbb{X}^*}, \phi(u)=\|u\|\}$ \cite{M}.  The symbol $\eta(\mathbb{X})$ is denoted as the numerical index of $\mathbb{X}$ which is defined as $\eta(\mathbb{X}):=\inf\{w(T): T \in \Li(\mathbb{X}), \|T\|=1\}.$ 	Let us now mention the  definition of $\rho$-orthogonality introduced in \cite{Mi}.  

\begin{definition}
	Let $\mathbb{X}$ be a normed linear space and let $x, y \in \mathbb{X}.$ The norm derivatives at $x$ in the direction of $y$ are defined as:
	\begin{eqnarray*}
		\rho'_{\pm}(x, y)&=& \|x\| \lim_{t \to 0^{\pm}} \frac{\|x+ty\|-\|x\|}{t}; \\
		\rho'(x, y) &=& \frac{1}{2}(\rho'_+(x, y) + \rho'_-(x, y)).
	\end{eqnarray*}
	We say that $x$ is $\rho$-orthogonal to $y$, i.e., $x \perp_{\rho} y$ if $\rho'(x, y)=0.$ 
\end{definition}
Note that $\rho$-orthogonality is real homogeneous, i.e., for any $\alpha, \beta \in \mathbb{R},$ $x \perp_\rho y \iff \alpha x\perp_\rho \beta y.$ 
 Let us mention some of the important results regarding the functions $\rho'_+$ and $\rho'_-.$ 

\begin{lemma}\label{functional}\cite[Th. 2.4]{Woj}
	Let $\mathbb{X}$ be a normed linear space. Then for $x, y \in S_\mathbb{X},$ 
	\begin{eqnarray*}
		\rho'_+(x, y) &:=& \|x\|\sup \{\Re(f(y)): f \in Ext(J(x))\},\\
		\rho'_-(x, y) &:=&\|x\|\inf \{\Re(f(y)): f \in Ext(J(x))\}.
	\end{eqnarray*}
\end{lemma}

\begin{lemma}\cite{Amir}\label{B-J}
	Let $\mathbb{X}$ be a normed linear space. Then $x \perp_{B} y$ if and only if $\rho'_{-}(x, y) \leq 0 \leq \rho'_{+}(x, y).$ 
\end{lemma}

For more on $\rho'_+, \,\rho'_-$ and related results  readers may see \cite{AST, CW, CW2, S, Woj19}. It is a well known fact from \cite{CW, CW2} that $\perp_{\rho} \subset \perp_B$ in any normed linear space $\mathbb{X}.$ Following \cite{GPS}, Given any $x \in \mathbb{X},$ we say that $x$ is \emph{$\rho$-left symmetric} (\emph{$\rho$-right symmetric})  if $x \perp_\rho y$ implies $y \perp_\rho x$ $(y\perp_\rho x \, \, \mbox{implies}\, \, x\perp_\rho y),$ for all $y \in \mathbb{X}.$ If $x$ is  both $\rho$-left and $\rho$-right symmetric then we say that $x$ is $\rho$-symmetric.

The paper is divided into two main parts. In the first, we establish necessary and sufficient conditions for $\rho$-orthogonality in $\Li(\mathbb{X}, \mathbb{Y})$ and obtain a connection between $\rho$-orthogonality and the shape of an operator’s maximal numerical range. The second part focuses on $\rho$-left and $\rho$-right symmetricity in $\Li(\mathbb{H})$. We completely characterize these symmetries on finite-dimensional Hilbert spaces, showing that in a real two-dimensional space, scalar multiple of orthogonal operators are the only nontrivial $\rho$-left (or $\rho$-right) symmetric ones. For $\dim \mathbb{H} \geq 3$, if $T \in \Li(\mathbb{H})$ attains its norm, only the zero operator is $\rho$-left symmetric. The same holds for diagonal operators regardless of norm attainment. Likewise, for $\rho$-right symmetry, the zero operator is the sole example when $3 \leq \dim \mathbb{H} < \infty$.

\section{$\rho$-orthogonality in $\Li(\mathbb{X}, \mathbb{Y})$}

We begin this section with the following proposition regarding the homogeneity of $\rho$-orthogonality. We omit the proof as it follows directly from the definition of norm derivatives. 

\begin{prop}\label{homogeneity}
	Let $x, y\in \mathbb{X}$ and $\lambda$ be a nonzero scalar. Then the following hold true:
	\begin{itemize}
		\item[(i)] If $\lambda \in \mathbb{R}$ then $\rho'(x, \lambda y)=\rho'(\lambda x, y)=\lambda \rho'(x, y).$
		\item[(ii)] If $\lambda \in \mathbb{C}$ then $ \rho'(x, \lambda y)=|\lambda|^2\rho'(\frac{1}{\lambda}x, y),$ and $ \rho'(\lambda x, y)=|\lambda|\rho'(x, \frac{1}{\lambda}y).$
		\item[(iii)] $\rho'(\lambda x, \lambda y)=|\lambda|^2\rho'(x, y),$ for any $\lambda\in \mathbb{C}$.
	\end{itemize} 
\end{prop}

	From the above proposition, we obtain that although for the real scalars, $\rho$-orthogonality is homogeneous but it is not true in case of complex scalars. On the other hand, we would like to mention that  using Proposition \ref{homogeneity}(ii) it is easy to observe that both $\rho$-left and $\rho$-right symmetric elements possess the homogeneity property. 
Let us now observe  the characterization of $\rho$-orthogonality via the abstract numerical range. 

\begin{prop}\label{ortho; rho}
	Let $T, A \in \Li(\mathbb{X}, \mathbb{Y})$. Then $T \perp_\rho A$ if and only if $$\sup \Re (\Omega(T, A))+\inf \Re (\Omega(T, A))=0,$$ where $\Omega(T, A):= \big\{\lim y_n^*(Ax_n): (x_n, y_n^*) \in S_\mathbb{X}\times S_{\mathbb{Y}^*}, \lim y_n^*(Tx_n)=\|T\|= 1\big\}.$
\end{prop}

\begin{proof}
	Note that $\rho'_+(T, A)= \|T\|\sup\{\Re \, (\phi(A)): \phi \in S_{\Li(\mathbb{X}, \mathbb{Y})^*}, \phi(T)=\|T\|\}.$ We note that the above set on which the supremum is taken is the  real part of the abstract numerical range of $A$  with respect to $(\Li(\mathbb{X}, \mathbb{Y}), T).$ Therefore,	we can write $\rho'_+(T, A)= \|T\|\sup \Re (V(\Li(\mathbb{X}, \mathbb{Y}), T, A)).$ From \cite[Th. 2.3]{M}, we obtain that $\Re(V(\Li(\mathbb{X}, \mathbb{Y}), T, A))= \Re(\conv \big\{ \Omega(T, A)\big\}).$  Since $\sup \Re \, (\Omega(T, A))= \sup \Re \{\conv \Omega(T, A)\},$ it follows that $\rho'_+(T, A)=\|T\| \sup \Re(\Omega(T, A)).$ Similarly, we get $\rho'_-(T, A)=\|T\|\inf\Re(\Omega(T, A)).$ This proves the desired result. 
\end{proof}

Next, we consider $\mathbb{X}$ to be a reflexive Banach space. Then for any $\mathbb{Y}$ we have the complete structure of the extreme points of the unit ball of $\K(\mathbb{X}, \mathbb{Y})$ \cite[Cor. 2.2]{Woj}. In particular,
\[\Ext B_{\K(\mathbb{X}, \mathbb{Y})^*}=\{y^*\otimes x\in \K(\mathbb{X}, \mathbb{Y})^*: x\in \Ext B_\mathbb{X}, y^*\in \Ext B_{\mathbb{Y}^*}\}.\] Using this we have the following nicer form of Proposition \ref{ortho; rho}. 
\begin{prop}
	Let $\mathbb{X}$ be a reflexive Banach space and $\mathbb{Y}$ be any normed linear space. Given any $T, A \in \K(\mathbb{X}, \mathbb{Y}),$  $T \perp_\rho A$ if and only if $$\sup \Re (\Omega(T, A))+\inf \Re (\Omega(T, A))=0,$$ where $\Omega(T, A):= \big\{ y^*(Ax): (x, y^*) \in \Ext B_\mathbb{X}\times \Ext B_{\mathbb{Y}^*},  y^*(Tx)=\|T\|\big\}.$
\end{prop}

In case of Hilbert space operators, this characterization is connected with some well-known  numerical range. We first recall the definition from \cite{Magajna}.
	\begin{definition}
	Given $T, A \in \Li(\mathbb{H}),$ the \emph{maximal numerical range of $A^*T$ relative to $T$,} denoted by $W_{T}(A^*T),$ is defined as the following set:
	\[W_{T}(A^*T):=\big\{\lim \langle A^*Tx_n, x_n\rangle : x_n \in S_\mathbb{X} \,\,\mbox{with}\, \|Tx_n\|\to \|T\|\big\}.\]
\end{definition}
We note from \cite[Lem. 2.1]{Magajna} that $W_T(A^*T)$ is a closed convex subset of $\mathbb{C}.$ We mention here that the well known Bhatia-\v Semrl Theorem can also be stated as:
\emph{ Given any $T, S \in \Li(\mathbb{H})$ $T \perp_B S$ if and only if $0 \in W_T(S^*T)$}  \cite[Lem. 2.2]{Magajna}.  In light of this statement we note the following observation regarding $\rho$-orthogonality in $\Li(\mathbb{H}).$ 

\begin{theorem}\label{cor;rho-symmetry}
	Let $T, A \in \Li(\mathbb{H}).$  Then $T\perp_\rho A$ if and only if the real part of the maximal numerical range of $A^*T$ relative to $T$ is symmetric with respect to origin.
\end{theorem}

\begin{proof}	
    From Proposition  \ref{ortho; rho} we have $T \perp_\rho A$ if and only if  $$\sup \Re(\Omega(T, A))+ \inf \Re(\Omega(T, A))=0,$$ where $\Omega(T, A)= \big\{\lim\langle Ax_n, y_n\rangle: \{x_n\}, \{y_n\} \subset S_\mathbb{H}, \lim\langle Tx_n, y_n\rangle = \|T\|\big\}.$ Observe that \[\Omega(T, A)= \big\{ \lim\langle Tx_n, Ax_n\rangle: \{x_n\} \subset S_\mathbb{H}, \lim \|Tx_n\|=\|T\|\big\}=W_T(A^*T).\] Since $W_T(A^*T)$ is a compact subset of $\mathbb{R},$ it follows that for each $\lambda \in W_T(A^*T)$ there exists $-\lambda \in W_T(A^*T).$ Therefore, $W_T(A^*T)=-W_T(A^*T).$	This proves our result.
    \end{proof}
 Using the above proposition we obtain a connection between \emph{symmetric numerical range} and $\rho$-orthogonality. We recall that for each $\theta\in [0, 2\pi),$  $L_{\theta}:=\{re^{i\theta}: r\in \mathbb{R}\}$ denotes the line passing through origin in the direction of $\theta.$ The symbol $\Pr_{\theta}(z)$ is the orthogonal projection of $z$ onto the line $L_{\theta}.$ In particular, $\Pr_{\theta}(z)=e^{i\theta}(\Re (z)\cos\theta+\Im(z)\sin\theta).$
    \begin{theorem}
    Let $A \in \Li(\mathbb{H})$. Then $W(A)$ is symmetric about the origin if and only if for each $\theta \in [0, 2\pi)$, $e^{i\theta}I\perp_\rho A.$ 
    \end{theorem}
    
    \begin{proof}
    	Observe that $W(A)$ is symmetric about the origin if and only if  for each $\theta\in [0, 2\pi),$ $\Pr_{\theta}(W(A))$ is symmetric about the origin. To prove the necessary part, suppose that for each $\theta\in [0, 2\pi),$ $\Pr_{\theta}(W(A))$ is symmetric about the origin. We note that for any $\lambda\in W(A),$ $e^{i\theta}\Pr_0(e^{-i\theta}\lambda)=\Pr_{\theta}(\lambda).$ Indeed, given any $\mu \in \Pr_0(W(e^{-i\theta}A)),$ $\mu= \Re (e^{-i\theta}\lambda),$ for some $\lambda\in W(A).$ Then it is straightforward to see that $$e^{i\theta}\mu = e^{i\theta}\Re (e^{-i\theta}\lambda)=e^{i\theta}(\Re (\lambda)\cos\theta+\Im(\lambda)\sin\theta)= \text{Pr}_{\theta}(\lambda).$$ Thus we obtain  $\Pr_0(W(e^{-i\theta}A))$ is symmetric about the origin. In other words, $$\max\Re \langle x, e^{-i\theta}Ax\rangle+ \min\Re \langle x, e^{-i\theta}Ax\rangle=0.$$ From Theorem \ref{cor;rho-symmetry} we obtain $I\perp_\rho e^{-i\theta}A.$ Now using Proposition \ref{homogeneity} (ii), we get $e^{i\theta}I\perp_\rho A.$\\
    	To prove the sufficient part, let $\lambda\in W(A).$ Then $e^{-i\theta}\lambda\in W(e^{-i\theta}A).$ Let $\mu=\Re(e^{-i\theta}\lambda)=\Pr_0(e^{-i\theta}\lambda).$ Since $I\perp_\rho e^{-i\theta}A$, it follows that there exists $\lambda'\in W(A)$ such that $\Re(e^{-i\theta}\lambda')=\Pr_0(e^{-i\theta}\lambda')=-\mu.$ From necessary part, we already have $e^{i\theta}\Pr_0(W(e^{-i\theta}A))=\Pr_{\theta}(W(A)).$ This implies $e^{i\theta}\mu = \Pr_{\theta}(\lambda)$ and $-e^{i\theta}\mu = \Pr_{\theta}(\lambda').$ Therefore, we obtain $\lambda'=-\lambda\in W(A),$ which completes the proof. 
    \end{proof}

In the next result we provide a necessary condition for $\rho$-orthogonality 	between two operators.

	\begin{theorem}\label{th}
		Let $\mathbb{X}$ be a reflexive Banach space and let $\mathbb{Y}$ be any normed linear space. Suppose that $T, A \in \K (\mathbb{X}, \mathbb{Y})$.  Then $T \perp_\rho A$ implies that there exist $x, y \in M_T$ such that $\rho'(Tx, Ax) \geq 0$ and $\rho'(Ty, Ay)\leq 0.$\\ In particular, suppose that $M_T=D \cup (-D)$ for some connected subset $D$ of $S_\mathbb{X}$. Then $T \perp_\rho A$ implies that there exists $x_0 \in D,$  $Tx_0 \perp_\rho Ax_0.$
	\end{theorem}
	
	\begin{proof}
	 Let $T \perp_\rho A.$ Then $\rho'(T, A)= \rho'_+(T, A)+ \rho'_-(T, A)=0.$ Following the expression from \cite[Th. 3.2]{Woj}, we note that 
	 \[\rho'_+(T, A)= \sup \{\rho'_+(Tx, Ax): x \in M_T \cap \Ext B_\mathbb{X}\},\]
	 \[\rho'_-(T, A)= \inf \{\rho'_-(Tx, Ax): x \in M_T \cap \Ext B_\mathbb{X}\}.\] For each $x \in M_T \cap \Ext B_\mathbb{X},$ $$\rho'_+(Tx, Ax)= \max_{z^* \in J(Tx)} \Re(z^*(Ax))= \Re(z_x^*(Ax)), \mbox{for some $z_x^*\in J(Tx)$}).$$ Thus we have $\rho'_+(T, A)= \sup \{\Re (z_x^*(Ax)): x \in M_T \cap \Ext B_\mathbb{X}\}.$ Consider a sequence $\{x_n\}_{n \in \mathbb{N}} \subset M_T \cap \Ext B_\mathbb{X}$ such that $\Re (z_{x_n}^*(Ax_n)) \longrightarrow \rho'_+(T, A).$ As $\mathbb{X}$ is reflexive and $A$ is compact, it follows that there exists $x_0 \in M_T \cap \Ext B_\mathbb{X}$ such that $\rho'_+(T, A)= \Re (z_{x_0}^*(Ax_0)),$ for some $z_{x_0}^*\in J(Tx_0).$ This implies $\rho'_+(T, A)=\rho'_+(Tx_0, Ax_0).$ Proceeding  similarly as above there exists $y_0 \in M_T \cap \Ext B_\mathbb{X}$ such that $\rho'_-(T, A)= \rho'_-(Ty_0, Ay_0).$ As $T \perp_\rho A,$ we get $\rho'_-(Ty_0, Ay_0)=-\rho'_+(Tx_0, Ax_0).$ Then one can observe that
	  \begin{eqnarray*}
	 	\rho'(Ty_0, Ay_0)&=& \frac{1}{2}\bigg(\rho'_+(Ty_0, Ay_0) + \rho'_-(Ty_0, Ay_0)\bigg)\\ &\leq& \frac{1}{2}\bigg(\rho'_+(Tx_0, Ax_0)+(-\rho'_+(Tx_0, Ax_0)\bigg)=0.
	 \end{eqnarray*}
	 Following same line of argument we can show that $\rho'(Tx_0, Ax_0)\geq 0.$

	  In particular, if $M_T= D \cup (-D)$ for some connected subset $D$ of $S_\mathbb{X},$ then the function $\phi: M_T \longrightarrow \mathbb{R} $ defined as $\phi(x)=\rho'(Tx, Ax)$ is a continuous function. By the above argument there exist $x, y \in M_T$ such that $\rho'(Tx, Ax)\geq 0$ and $\rho'(Ty, Ay)\leq 0.$ Without loss of generality we may assume that $x, y \in D.$ Since $D$ is connected, by the continuity of $\phi$ we obtain that there exist $x_0\in D$ such that $\rho'(Tx_0, Ax_0)=0.$ 
	\end{proof}
As a consequence of Theorem \ref{th} the following corollary is immediate.
	
	\begin{cor}\label{cor; unique}
		Let $T, A \in \Li(\mathbb{X}, \mathbb{Y})$ with $M_T=\{\mu x_0: |\mu|=1\}$ for some $x_0\in S_\mathbb{X}$. Then $T\perp_\rho A$ if and only if $\rho'(Tx_0, Ax_0)=0.$
	\end{cor}
	In the following example we see that unlike Birkhoff-James orthogonality, the above necessary condition for $T \perp_\rho A$ is not sufficient even if we consider $M_T=D \cup (-D)$, where $D$ is a connected subset of $S_\mathbb{X}.$ 
	\begin{example}
		Let us consider $T, A\in \Li(\ell_\infty^2(\mathbb{R}))$ are defined by $T (1, 1)=(1, \frac{1}{2})$ and $T(1, -1)=(1, -\frac{1}{2})$ whereas $A(1, 1)=(\frac{1}{2}, 0)$ and $A(1, -1)=(-1, 0).$ Observe that $M_T=F \cup (-F),$ where $F=\{(1, x): -1 \leq x \leq 1\}$ is a connected subset of $S_{\ell_\infty^2}.$ Note that $A(1, \frac{1}{3})=0$ and so $T(1, \frac{1}{3}) \perp_\rho A(1, \frac{1}{3}).$ On the other hand, by a straightforward computation $\max\{\rho'_+(Tx, Ax): x \in M_T\cap \Ext(B_{\ell_\infty^2})\}=\frac{1}{2}$ and $\min \{\rho'_-(Tx, Ax): x \in M_T\cap \Ext(B_{\ell_\infty^2})\}=-1.$ Hence $T \not\perp_\rho A.$  
	\end{example}

	For the sufficient part in general we have the following result.
	
	\begin{prop}\label{sufficient; ortho}
		Let $T, A \in \Li(\mathbb{X}, \mathbb{Y})$ and let $M_T \neq \emptyset.$ If $Tx \perp_\rho Ax$ for all $x \in M_T$, then $T \perp_\rho A.$ 
	\end{prop}
	\begin{proof}
		As we know $\rho'_+(T, A)= \sup_{x \in M_T} \rho'_+(Tx, Ax),$ then there exists $\{x_n\} \subset M_T$ such that $\rho'_+(Tx_n, Ax_n) \longrightarrow \rho'_+(T, A).$ This implies $-\rho'_+(Tx_n, Ax_n) \longrightarrow -\rho'_+(T, A).$ Since $Tx_n \perp_\rho Ax_n,$ for each $n \in \mathbb{N},$ we have $\rho'_-(Tx_n, Ax_n) \longrightarrow -\rho'_+(T, A).$ We claim that $-\rho'_+(T, A)= \rho'_-(T, A).$ If possible let $\rho'_-(T, A) < -\rho'_+(T, A).$ This implies there exists a $z \in M_T$ such that $\rho'_-(Tz, Az)< -\rho'_+(T, A).$ As $\rho'_-(Tz, Az)=-\rho'_+(Tz, Az),$ we get $\rho'_+(Tz, Az)> \rho'_+(T, A),$ which is a contradiction. So, our claim is established. Therefore, $\rho'_+(T, A)+\rho'_-(T, A)=0,$ i.e., $T \perp_\rho A.$ This completes the proof.
	\end{proof}

	In the following example we observe that  the sufficient condition in Proposition \ref{sufficient; ortho} is not necessary.
	
	\begin{example}
		Let $T: \ell_\infty^2(\mathbb{R}) \longrightarrow \ell_\infty^2(\mathbb{R})$ and  $A: \ell_\infty^2(\mathbb{R}) \longrightarrow \ell_\infty^2(\mathbb{R})$  be defined as $T(1, 1)=(1, 0)$, $T(-1, 1)=(\frac{1}{2}, 1)$ and  $A(1, 1)=(1, 0), A(-1, 1)=(0, -1),$ respectively. Observe that $M_T\cap \Ext B_{\ell_\infty^2}=\{(\pm(1, 1), \pm (-1, 1))\}.$ Now $\rho'_+(T, A)= \max\{\rho'_+(Tx, Ax): x \in M_T \cap \Ext B_{\ell_\infty^2}\}=1,$ whereas $\rho'_-(T, A)=\min\{\rho'_+(Tx, Ax): x \in M_T \cap \Ext B_{\ell_\infty^2}\}=-1.$ This implies $T \perp_\rho A.$ But note that $T(1, 1) \not\perp_\rho A(1, 1)$.
	\end{example}

	


\section{$\rho$-symmetricity in $\Li(\mathbb{H}).$}	
In this section our aim is to investigate the $\rho$-left and $\rho$-right symmetric operators defined on a Hilbert space.  
  \subsection{$\rho$-left symmetry:}  First we observe the following necessary condition for a  $\rho$-left symmetric operator.	
		\begin{prop}\label{prop; perp zero}
		Let $T \in \Li(\mathbb{H})$ be a nonzero $\rho$-left symmetric point. Suppose that $M_T\neq \emptyset.$ Then $Ty=0$ for all $y \in (span M_T)^\perp.$
	\end{prop}
	
	\begin{proof}
		Suppose that $M_T = S_{H_0},$ where $H_0$ is a subspace of $\mathbb{H}$  \cite[Th. 2.2]{SP}. Firstly, note that if $H_0=\mathbb{H}$ then $H_0^{\perp}=0$ and thus we have nothing to show. Suppose on the contrary assume that  there exists $z \in H_0^\perp$ such that $Tz \neq 0.$ Define $A \in \Li(\mathbb{H})$ as $Az = Tz$ and $Aw=0,$ for all $w \in z^\perp.$ Note that $H_0 \subset z^\perp.$  So, $Ax=0, $ for all $x \in H_0$ and therefore, $\langle Tx, Ax \rangle =0,$ for all $x \in M_T.$ Clearly, then using Proposition \ref{ortho; rho} we get $T \perp_\rho A.$ On the other hand, we note that $M_A = \{\mu z: |\mu|=1\}.$ Thus we get $\Re (\langle Az, Tz \rangle) = \|Tz\|^2 \neq 0.$ This yields $A \not\perp_\rho T,$ which contradicts the fact that $T$ is $\rho$-left symmetric.
	\end{proof}

	

	The following lemma shows that the $\rho$-orthogonality in $\Li(\mathbb{H})$ is invariant under unitary equivalence.
	
	\begin{lemma}\label{lem}
		Let $T, A \in \mathcal{L}(\mathbb{H}).$ For any unitary $U \in \mathcal{L}(\mathbb{H}),$ we have 
		\[ U^*TU \perp_\rho U^*AU \iff T \perp_\rho A.\]
	\end{lemma}
	 \begin{proof}
	 	It is sufficient to show that $\rho_{\pm} (T, A) = \rho'_{\pm} (U^*TU, U^*AU).$ For that we only show $\rho'_{+}(T, A) = \rho'_+(U^*TU, U^*AU)$ as the case for $\rho'_-$ follows similarly. Without loss of generality we assume that $\|A\|=\|T\|=1.$ Note that 
	 	\begin{eqnarray*}
	 		 \rho'_+(U^*TA, U^*AU)&=& \lim_{t\to 0^+} \frac{\|U^*TU + tU^*AU\|-\|U^*TU\|}{t}\\
	 		 &=& \lim_{t\to 0^+} \frac{\|U^*(T + tA)U\|-\|U^*TU\|}{t}\\
	 		 &=& \lim_{t\to 0^+} \frac{\|T + tA\|-\|T\|}{t}\\
	 		 &=& \rho'_+(T, A).
	 	\end{eqnarray*} 
	 	This proves the lemma.
	 \end{proof}


In the next result we characterize the $\rho$-left symmetric operators on a two-dimensional real Hilbert space.	
	\begin{theorem}\label{prop}
		Let $\mathbb{H}$ be a two-dimensional real Hilbert space and let $T \in \Li(\mathbb{H}).$ Then $T$ is $\rho$-left symmetric if and only if  $T$ is a scalar multiple of isometry. 
	\end{theorem}
	 
	 \begin{proof}
	 	We first prove the sufficient part. For this we show that the identity operator is $\rho$-left symmetric. Let $I \perp_\rho A,$ for some $A \in \mathcal{L}(\mathbb{H}).$ According to Schur's theorem we obtain an orthogonal equivalent matrix $B \in \Li(\mathbb{H})$ such that it has one of the following forms:
	\begin{eqnarray*}
	 		(i) \, B= U^t AU=
	 			\begin{bmatrix}
	 				\lambda_1 & a \\
	 				0 & \lambda_2 
	 			\end{bmatrix}
\mbox{or,}\quad 
	 				(ii)\,  B= U^tAU= 
	 			\begin{bmatrix}
	 				p & q\\
	 				-q & p
	 			\end{bmatrix}
	 		 \end{eqnarray*} 
	 		 	 for some orthogonal matrix $U \in \Li(\mathbb{H}).$ From Lemma \ref{lem} we have $I \perp_\rho B.$ This implies that \begin{equation}
	 		 	 	\max_{x \in S_{\mathbb{R}^2}} \langle x, Bx \rangle + \min_{x \in S_{\mathbb{R}^2}}\langle x, Bx \rangle =0.
	 		 	 	\end{equation}
	 		 	  Suppose that $B$ is of the form $(i).$   In other words, we have $W(B) = -W(B).$  Now observe that $$ C=B - \frac{\lambda_1+\lambda_2}{2}I = \begin{bmatrix} 
	 \frac{\lambda_1-\lambda_2}{2} & a\\
	 0 & \frac{\lambda_2 - \lambda_1}{2} 
	 \end{bmatrix}. $$ By some straightforward computation one can see that $W(C)= - W(C).$ Therefore, $W(B)-\frac{\lambda_1+\lambda_2}{2}= -W(B)+ \frac{\lambda_1+\lambda_2}{2}= W(B)+ \frac{\lambda_1+\lambda_2}{2}.$ Thus we obtain $\lambda_1 = -\lambda_2.$  Assume that $B= \begin{bmatrix}
	 \lambda & a\\
	 0 & -\lambda
	 \end{bmatrix}.$ If $a=0,$ then $M_B=S_{\mathbb{R}^2}$, which gives $B \perp_\rho I$ using (3.1). Suppose that $a \neq 0.$  Then from \cite[Prop. 3]{Hamed}, we note that $\langle Bx, x \rangle=0,$ for all $x \in M_B.$ Thus we get $B \perp_\rho I.$ On the other hand, whenever $B$ has the the form $(ii)$, we note that $B$ is an isometry. Thus $M_B=S_{\mathbb{R}^2}$ and therefore from (3.1), $\max_{x\in S_{\mathbb{R}^2}}\langle Bx, x \rangle+\min_{x \in S_{\mathbb{R}^2}}\langle Bx, x \rangle=0,$ i.e., $B \perp_\rho I.$  Therefore, we have $A \perp_\rho I.$ This implies that $I$ is $\rho$-left symmetric.  Consequently, this proves that any isometry in $M_2$ is $\rho$-left symmetric.
	 
	 To prove the necessary part, suppose on the contrary that $T$ is not an isometry. Then $M_T$ is singleton (up to the sign). This implies that $T$ is smooth. Let us consider $T \perp_B A$ for some $A \in M_2(\mathbb{R}).$ Since $T$ is smooth, we have $T\perp_\rho A.$  From hypothesis we have $T$ is $\rho$-left symmetric and therefore, $A \perp_\rho T$ which implies $A \perp_B T.$ This shows that $T$ is left symmetric with respect to Birkhoff-James orthogonality. Now applying \cite[Th. 2.1]{SGP17} we get $T=0$, a contradiction. This completes the proof.
	\end{proof}
	
	For Hilbert space, the identity operator, $I\in \Li(\mathbb{H}),$ is $\rho$-left symmetric within the class of self-adjoint operators. 
	 
	 \begin{prop}\label{self adjoint}
	 	Let $A \in \Li(\mathbb{H})$ be a self-adjoint operator. Then $I \perp_\rho A \implies A \perp_\rho I.$  
	 \end{prop}

	 \begin{proof}
	 	 Let $I \perp_\rho A.$ This implies $\sup_{x\in S_\mathbb{H}} \Re\langle x, Ax\rangle + \inf_{x\in S_\mathbb{H}} \Re\langle x, Ax \rangle=0.$ As $W(A)$ is real and $w(A)=\|A\|$,  $W(A)=[-\|A\|, \|A\|].$ Now there exists $\{x_n\} \subset S_\mathbb{H}$ such that $\lim|\langle Ax_n, x_n\rangle |=\|A\|.$ Then clearly, $\lim\|Ax_n\|=\|A\|,$ which means $\{x_n\}$ is a norming sequence of $A.$ Similarly, there exists a norming sequence $\{y_n\}$ such that $\lim\langle Ay_n, y_n \rangle=-\|A\|. $ Therefore, from Proposition \ref{ortho; rho} we obtain $A \perp_\rho I.$ 
	 \end{proof}
	 
	 From the above result we note the following remark in the Banach space setting:
	 
	 \begin{remark}
	 	If we consider a real Banach space $\mathbb{X}$ with $\eta(\mathbb{X})=1,$ then for each $A\in \Li(\mathbb{X}),$ $w(A)=\|A\|.$ Then following the argument from Proposition \ref{self adjoint} we obtain that $I \in \Li(\mathbb{X})$ is $\rho$-left symmetric.  Indeed, as $I\perp_\rho A$, we get $W(A)=[-\|A\|, \|A\|].$ This implies there exists $(x_n, x_n^*), (y_n, y_n^*)\in S_\mathbb{X}\times S_{\mathbb{X}^*}$ satisfying $x_n^*(x_n)=1$ and $y_n^*(y_n)=1$ for each $n\in \mathbb{N}$ such that $x_n^*(Ax_n)\to \|A\|$ and $y_n^*(Ay_n)\to -\|A\|.$ Taking two norming sequence $(x_n, x_n^*)$ and $(y_n, -y_n^*)$ we can conclude that $\sup \Omega(A, I)=1$ and $\inf \Omega(A, I)=-1.$ Therefore, $A\perp_\rho I.$ This shows that whenever $\eta(\mathbb{X})=1$, we get for any $A\in \Li(\mathbb{X})$,  numerical range of $A$ is symmetric about the origin implies the maximal numerical range of $A$ is also symmetric about the origin.
	 \end{remark}

In the next result we show that there does not exist any nonzero $\rho$-left symmetric norm attaining operators, provided that $\dim \mathbb{H}\geq 3$. 
\begin{theorem}\label{left; main}
	Let $\dim\mathbb{H}\geq 3$ and let $T\in \Li(\mathbb{H})$ be such that $M_T \neq \emptyset.$ Then $T$ is $\rho$-left symmetric if and only if $T=0.$ 
\end{theorem}
\begin{proof}
Suppose on the contrary that $T$ is nonzero and $\|T\|=1$. We prove the necessary part of the theorem by considering the following two cases.\\
	\textbf{Case-I:} Suppose that $M_T=S_\mathbb{H}.$ Let $\{x_\alpha: \alpha\in \Lambda\}$ be an orthonormal basis for $\mathbb{H}.$ For some $\alpha_1, \alpha_2, \alpha_3 \in \Lambda,$ let us define $A \in \Li(\mathbb{H})$ as the following:
	\begin{eqnarray*}
	  Ax_{\alpha_1}&=& \frac{1}{\sqrt{2}}(Tx_{\alpha_1}+Tx_{\alpha_2}),\\
	  Ax_{\alpha_2}&=& \frac{1}{\sqrt{2}}(Tx_{\alpha_2}-Tx_{\alpha_1}),\\
	  Ax_{\alpha_3}&=&-\frac{1}{\sqrt{2}}Tx_{\alpha_3},\\
	  Ax_{\alpha}&=&0, \quad \mbox{for all} \,\, \alpha \in \Lambda \setminus \{\alpha_1, \alpha_2, \alpha_3\}.
	\end{eqnarray*}
	For any $z\in S_\mathbb{H},$ we can write $z=\sum c_{i} x_{\alpha_i},$ where $\sum |c_i|^2=1$. Then it is straightforward to see that $\|Az\|^2= |c_1|^2 + |c_2|^2 + \frac{1}{2}|c_3|^2 \leq 1.$ Clearly, $z\in M_A$ iff $z \in span\{x_{\alpha_1}, x_{\alpha_2}\}\cap S_\mathbb{H}.$  As $M_T=S_\mathbb{H}$, for any distinct $\alpha, \beta \in \Lambda,$ we have $\langle Tx_{\alpha}, Tx_{\beta}\rangle =0.$ Now for any $z\in S_\mathbb{H},$ one can compute that $\Re \langle Tz, Az\rangle= \frac{1}{\sqrt{2}}(|c_1|^2+ |c_2|^2- |c_3|^2). $  This shows that $\max_{z\in M_T} \Re \langle Tz, Az\rangle =\frac{1}{\sqrt{2}} $ and $\min_{z\in M_T} \Re \langle Tz, Az\rangle =-\frac{1}{\sqrt{2}}.$ Therefore, we have $\rho'(T, A)=0,$ i.e., $T\perp_\rho  A.$ One the other hand, one can observe that $\Re \{\langle Az, Tz \rangle: z\in M_A\}=\frac{1}{\sqrt{2}}.$ This implies that $\rho'(A, T)\neq 0,$ i.e., $A\not\perp_\rho T.$ This is a contradiction. Thus we obtain that isometry in $\Li(\mathbb{H})$ is not $\rho$-left symmetric.\\
	\textbf{Case-II:} Suppose that $M_T =S_{H_0}\subsetneq S_\mathbb{H},$ where $H_0$ is a proper subspace of $\mathbb{H}.$ Note that $\mathbb{H}=H_0 \oplus H_0^{\perp}.$ Moreover, as $T$ is $\rho$-left symmetric, $T(H_0^{\perp})=0.$ Let us consider $\{x_\alpha: \alpha \in \Lambda\}$ to be an orthonormal basis for $H_0$ and $\{x_\beta: \beta \in \Lambda'\}$  an orthonormal basis for $H_0^{\perp}$. Thus $\{x_\alpha, x_\beta: \alpha\in \Lambda, \beta \in \Lambda'\}$ is an orthonormal  basis for $\mathbb{H}.$ First of all observe that if $T(H_0)\subset H_0$ then $T$ is an isometry on the subspace $H_0.$ In other words, $T$ can be written as $T= \begin{pmatrix}
		T_1 & 0\\
		0 & 0
	\end{pmatrix},$ where $T|_{H_0}=T_1.$ Clearly, $T_1$ is an isometry on $H_0.$ From Case-I, we already can see that $T$ can not be $\rho$-left symmetric unless $T=0.$ So, assume that $T(H_0) \not\subset H_0.$ Then there exists $\alpha_0\in\Lambda, \beta_0 \in \Lambda'$ such that $x_{\alpha_0}\in H_0, x_{\beta_0}\in H_0^{\perp}$ satisfying $\langle Tx_{\alpha_0}, x_{\beta_0}\rangle \neq 0.$ Now we define $A\in \Li(\mathbb{H})$ as:
	\begin{eqnarray*}
		Ax_{\alpha_0}&=& w_0, \quad \mbox{where} \,\, w_0\perp Tx_{\alpha_0},\\
		Ax_{\beta_0} &=& \frac{1}{\sqrt{2}}(w_0+ Tx_{\alpha_0}),\\
		Ax_{\alpha}&=&0, \quad \forall \alpha\in \Lambda.
	\end{eqnarray*}
	From this it is a simple computation to observe that $M_A= \{\frac{\mu}{\sqrt{2}}(x_{\alpha_0}+ x_{\beta_0}): |\mu|=1\}.$ Note that for each $x_\alpha\in H_0$, $\langle Tx_\alpha, Ax_\alpha\rangle=0.$ Therefore, $\rho'(T, A)=0,$ i.e., $T\perp_\rho A$ whereas, $$\langle A(\frac{1}{\sqrt{2}}(x_{\alpha_0}+ x_{\beta_0})), T(\frac{1}{\sqrt{2}}(x_{\alpha_0}+ x_{\beta_0}))\rangle= \frac{1}{2\sqrt{2}}\|Tx_{\alpha}\|^2\neq 0$$ implying that $A \not\perp_\rho T,$ which is a contradiction to the fact that $T$ is $\rho$-left symmetric. This completes the proof of the theorem.
	\end{proof}

As every compact operator on a Hilbert space always attains its norm, from Theorem \ref{left; main}, we have the following corollary.
\begin{cor}\label{compact}
	Let $\dim\mathbb{H}\geq 3$ and let $T\in \K(\mathbb{H})$. Then $T$ is $\rho$-left symmetric if and only if $T=0.$
\end{cor}

Further we extend our investigation for the \emph{diagonal operator} which may or may not attain its norm. Given a separable Hilbert space $\mathbb{H}$ if there exists an orthonormal basis $\{e_n\}$ of $\mathbb{H}$ such that $T\in \Li(\mathbb{H})$ satisfies $Te_n=\lambda_ne_n$ for each $n\in \mathbb{N}$ and $\lambda_n\in \mathbb{C},$ then we call $T$ a \emph{diagonal operator}.

\begin{theorem}\label{diagonal; infinite}
	Let $\mathbb{H}$ be separable                                                                                                                                                                                                                                               Hilbert space and let $T\in \Li(\mathbb{H})$ be diagonal. Then $T$ is $\rho$-left symmetric if and only if $T=0.$
\end{theorem}
\begin{proof}
	Note that if $M_T\neq \emptyset,$ then following Theorem \ref{left; main}, we get our desired result. So let us assume that $M_T=\emptyset.$ Let $T$ be nonzero with $\|T
	\|=1$. Suppose that $\{e_n\}_{n\in \mathbb{N}}$ is the standard orthonormal basis for $\mathbb{H}$ and $Te_k=\lambda_ke_k,$ for all $k\in \mathbb{N}.$ Without loss of generality assume that $|\lambda_1|\neq 0.$  Let us consider the sequence $\{\lambda_n\}$ such that $|\lambda_n|\to 1.$ Define $A\in \Li(\mathbb{H})$ as: $$ Ae_k=\frac{\lambda_k}{k}e_k, \forall k \in \mathbb{N}.$$ For any norming sequence $z_n$ of $A,$ we write $z_n=\sum_{k=1}^\infty \alpha_{n, k} e_k,$ where for each $n,$ $\sum_{k} |\alpha_{n, k}|^2=1.$ We want to show that for every $\epsilon>0,$ $|\langle Tz_n, Az_n\rangle |< C\epsilon,$ for some fixed $C>0.$  Given $\epsilon >0,$ consider $N\in \mathbb{N}$ such that $\sup_{k\geq N} \frac{|\lambda_k|}{|k|} < \epsilon.$ Now
	\begin{eqnarray*}
		|\langle Tz_n, Az_n\rangle|&=& |\langle \sum \alpha_{n, k}Te_k, \sum \alpha_{n, k}Ae_k\rangle|\\
		&=&|\sum |\alpha_{n, k}|^2\langle Te_k, Ae_k\rangle|\\
		&\leq&  \sum_{k=1}^{N-1} |\alpha_{n, k}|^2\frac{|\lambda_k|^2}{|k|} + \sum_{k\geq N} |\alpha_{n, k}|^2\frac{|\lambda_k|^2}{|k|}.		
	\end{eqnarray*}
	Note that since $\sup|\lambda_k|=1$ and $\sum_{k\geq N} |\alpha_{n, k}|^2\le 1,$ it follows that $$\sum_{k\geq N} |\alpha_{n, k}|^2\frac{|\lambda_k|^2}{|k|} \leq \sup|\lambda_k| \sum_{k\geq N} |\alpha_{n, k}|^2\frac{|\lambda_k|}{|k|}< \epsilon.$$ Now we observe the finite sum. Let us consider $\max_{1\leq k\leq N-1} |\lambda_k|^2=m<1.$ Also, assume that $s_n=\sum_{k=1}^{N-1} |\alpha_{n, k}|^2.$ Then we get $\sum_{k=1}^{N-1} |\alpha_{n, k}|^2\frac{|\lambda_k|^2}{|k|}\leq ms_n.$ Observe that $$\|Tz_n\|^2= \sum_{k=1}^{N-1} |\alpha_{n, k}|^2|\lambda_k|^2+ \sum_{k\geq N} |\alpha_{n, k}|^2|\lambda_k|^2 \leq ms_n+1.(1-s_n)= 1-(1-m)s_n.$$ This implies $1-\|Tz_n\|^2\geq (1-m)s_n.$ For sufficiently large $n,$ we obtain that $s_n \leq \epsilon$ as $m\neq1.$ Therefore, 
	$$|\langle Tz_n, Az_n\rangle|\leq \sum_{k=1}^{N-1} |\alpha_{n, k}|^2\frac{|\lambda_k|^2}{|k|}+ \sum_{k\geq N} |\alpha_{n, k}|^2\frac{|\lambda_k|^2}{|k|}< \epsilon(m+1).$$ Thus for each norming sequence $\{z_n\}$ of $A$, we have $\langle Tz_n, Az_n\rangle \to 0.$ From Theorem \ref{cor;rho-symmetry} we get $T\perp_\rho A.$ On the other hand, note that $\|A\|=|\lambda_1|$ and $M_A=\{\mu e_1: |\mu|=1\}.$ Therefore, $\langle Ae_1, Te_1\rangle =|\lambda_1|^2 \neq 0$ implies that $A\not\perp_\rho T.$ This contradicts the fact that $T$ is $\rho$-left symmetric. Therefore, $T=0$, which completes the proof of the theorem.
	 \end{proof}
\subsection{$\rho$-right symmetry:}
Next we turn our attention towards $\rho$-right symmetric operators in $\Li(\mathbb{H}).$ First we observe the following result.
	\begin{theorem}\label{th; right unitary}
		Let $\mathbb{H}$ be a real Hilbert space with $\dim \mathbb{H}=2$ and let $T \in \mathcal{L}(\mathbb{H}).$  $T$ is $\rho$-right symmetric if and only if $T$ is a scalar multiple of isometry. 
	\end{theorem}
	
	\begin{proof}
		To prove the sufficient part we only need to show that the identity operator $I\in \Li(\mathbb{H})$ is $\rho$-right symmetric. Let $A \perp_\rho I,$ for some $A \in \mathcal{L}(\mathbb{H})$. Without loss of generality assume that $\|A\|=1.$ Note that either $M_A= S_\mathbb{H}$ or $M_A = \{\pm x_0\},$ for some $x_0 \in S_\mathbb{H}.$ If $M_A= S_\mathbb{H}$ then $A \perp_\rho I$ implies that  $W(A)= W(-A).$ This shows from Theorem \ref{cor;rho-symmetry} that $I \perp_\rho A.$ Now let us assume that $M_A=\{\pm x_0\}.$ So $A \perp_\rho I$ implies $\langle Ax_0, x_0 \rangle =0.$ Let $Ax_0=y_0.$ Since $y_0 \neq 0,$ it follows that $x_0 \perp_B y_0.$ We claim that $\langle Ay_0, y_0 \rangle =0.$ If possible let $\langle Ay_0, y_0 \rangle \neq 0.$ For any $z \in S_\mathbb{H},$ we can write $z = \alpha x_0+ \beta y_0,$ where $|\alpha|^2 + |\beta|^2 =1.$  Now 
		\begin{eqnarray*}
		\|Az\|^2 &=& \langle \alpha Ax_0 +\beta Ay_0, \alpha Ax_0 +\beta Ay_0\rangle \\
		&=& \langle \alpha y_0 +\beta Ay_0, \alpha y_0 +\beta Ay_0\rangle\\
		&=& |\alpha|^2 \|y_0\|^2 + |\beta|^2 \|Ay_0\|^2 + 2 \alpha \beta \langle Ay_0, y_0\rangle\\
		&=& 1 + 2\alpha \beta \langle Ay_0, y_0\rangle - \beta^2(1- \|Ay_0\|^2).
		\end{eqnarray*}
		First suppose that $1-\|Ay_0\|^2 < |\langle Ay_0, y_0\rangle|.$ Then we take $\alpha = \frac{1-\|Ay_0\|^2}{\langle Ay_0, y_0\rangle}$ and $\beta = \sqrt{1-\alpha^2}.$ From the above equation we get that $$\|Az\|^2 = 1+ (1-\|Ay_0\|^2)(2\sqrt{1-\alpha^2}-(1-\alpha^2)).$$ Clearly, this implies that $\|Az\|^2 > 1,$ which is not possible. Now suppose that $1-\|Ay_0\|^2 \geq |\langle Ay_0, y_0\rangle|.$ From this we obtain that $$ |\langle Ay_0, y_0\rangle|^2 \leq \|Ay_0\|^2 \leq 1-|\langle Ay_0, y_0\rangle| < 1- |\langle Ay_0, y_0\rangle|^2 < \sqrt{1-|\langle Ay_0, y_0\rangle|^2}.$$ Therefore, we get $|\langle Ay_0, y_0\rangle| < \frac{1}{2^{\frac{1}{4}}}. $ Now we consider $\alpha=\sqrt{1-\langle Ay_0, y_0\rangle^2}$ and therefore, $\beta=\langle Ay_0, y_0\rangle.$ Clearly, $0 < \alpha, \beta <1.$ Now we put $\alpha $ and $\beta $ in the following expression:
		\begin{eqnarray*}
			2\alpha \beta \langle Ay_0, y_0\rangle - \beta^2(1- \|Ay_0\|^2)&=& 2\langle Ay_0, y_0\rangle^2\bigg(2\sqrt{1-\langle Ay_0, y_0\rangle^2} -(1-\|Ay_0\|^2)\bigg)\\
			&>& \langle Ay_0, y_0\rangle^2 \bigg(2\sqrt{1-\frac{1}{2^{\frac{1}{2}}}} - (1-\|Ay_0\|^2)\bigg) \\
			&>& 0. \quad \quad  (\mbox{As} \, \,  1-\|Ay_0\|^2<1)
		\end{eqnarray*}
		 This contradicts the fact that $\|A\|=1.$ Thus we get $\langle Ay_0, y_0\rangle=0.$ So,  $W(A)=\{\langle Az, z\rangle : x \in S_\mathbb{H}\} = \{\alpha \beta (\langle Ax_0, y_0\rangle + \langle Ay_0, x_0\rangle): \alpha^2+\beta^2=1\}.$ Clearly, this implies $W(A)= W(-A)$ and so $I \perp_\rho A.$ Thus for any isometry $T,$ it is easy to see that $T \perp_\rho A \implies A \perp_\rho T.$\\
		
		Next we prove the necessary part. Suppose on the contrary $T\in \Li(\mathbb{H})$ is not an isometry. Let  $M_T=\{\pm x_0: x_0\in S_\mathbb{H}\}.$ Assume that $z_0\in S_\mathbb{H}$ such that $x_{0}^{\perp} = span\{z_0\}.$ If $Tz_0=0$ then we consider $A\in \Li(\mathbb{H})$ such that $Ax_0=\frac{Tx_0}{2\|T\|^2}$ and $Az_0=z_0.$ Then clearly, $M_A=\{\pm z_0: |\mu|=1\}$ and $\langle Tz_0, Az_0\rangle =0.$ Thus we get $A\perp_\rho T.$ But $\langle Tx_0, Ax_0\rangle = \frac{1}{2}\neq 0$ and so $T\not \perp_\rho A.$ Next, suppose that $Tz_0\neq 0.$ Let $w_0(\neq 0)\in \mathbb{H} $ be such that $w_0\perp Tz_0.$ We consider $A\in \Li(\mathbb{H})$ defined as $Ax_0=Tx_0$ and $Az_0 = rw_0,$ where $r$ is a positive scalar such that $\|Tx_0\|< r\|w_0\|.$ Following the definition of $A$, one can see that $M_A=\{\pm z_0:\}.$ As $w_0\perp Tz_0,$ $\langle Az_0, Tz_0\rangle = 0.$ This gives $A \perp_\rho T,$ whereas, $\langle Tx_0, Ax_0\rangle =\|T\|^2 \neq 0$ implies $T\not\perp_\rho A. $ This contradiction proves the necessary part. Hence the theorem. 
	\end{proof}
	
As a corollary, we note the following result  related to the geometry of numerical range and the maximal numerical range of an operator.
	\begin{cor}\label{cor; sym}
		Let $\mathbb{H}$ be a real Hilbert space and let $\dim \mathbb{H}=2.$ Then for any $A\in \Li(\mathbb{H}),$ the numerical range of $A$ is symmetric with respect to the origin if and only if the maximal numerical range is also symmetric with respect to the origin.
	\end{cor}
	\begin{proof}
		One can observe that the necessary part is true if and only if the identity operator $I\in \Li(\mathbb{H})$ is $\rho$-left symmetric, whereas the sufficient part holds true if and only if $I$ is $\rho$-right symmetric.
		In case of the two-dimensional real Hilbert space, Theorems \ref{prop} and \ref{th; right unitary} demonstrate that only isometries satisfy these conditions. As the identity operator is isometry, the result holds true.
	\end{proof}
	From Theorem \ref{left; main}, we note that the necessary part of the above corollary is not true, whenever $\dim\mathbb{X}\geq 3$. Moreover, in the next theorem we observe that the sufficient part of the above corollary fails to hold whenever $\dim \mathbb{H}\geq 3.$
	
	\begin{theorem}\label{right; isometry}
		Let $\mathbb{H}$ be a Hilbert space with $\dim \mathbb{H} \geq 3$. Let $T \in \mathcal{L}(\mathbb{H})$ be an isometry. Then $T$ is not $\rho$-right symmetric.
	\end{theorem}
\begin{proof}
	Suppose that $\{x_\alpha: \alpha\in \Lambda\}$ is an orthonormal basis for $\mathbb{H}.$ Let us consider $A \in \mathcal{L}(\mathbb{H})$ such that $$Ax_{\alpha_1}= Tx_{\alpha_2}, Ax_{\alpha_2}=-Tx_{\alpha_1} \, \, \mbox{and} \, \, Ax_\alpha=\frac{1}{2}Tx_\alpha, \forall \alpha\in \Lambda \setminus \{\alpha_1, \alpha_2\}.$$ We claim that $M_A=span \{x_{\alpha_1}, x_{\alpha_2}\}\cap S_\mathbb{H}.$ Indeed, let $z \in S_\mathbb{H},$ be such that $z= \sum c_{\alpha_i} x_{\alpha_i},$ Since $Tx_{\alpha} \perp Tx_{\beta},$ for all $\alpha \neq \beta,$ it follows that  $$\|Az\|^2 = c_{\alpha_1}^2+c_{\alpha_2}^2 + \frac{1}{4}\sum_{i\neq 1, 2} c_{\alpha_i}^2 \leq 1.$$ Clearly, in the above inequality,  equality holds if and only if $c_{\alpha_i}=0,$ for all $i \neq 1, 2.$  So, $M_A=span \{x_{\alpha_1}, x_{\alpha_2}\}\cap S_\mathbb{H},$ as claimed. Now for any $v \in M_A,$ we write $v = c_1 x_{\alpha_1} + c_2x_{\alpha_2}$ with $c_1^2+c_2^2 =1.$ Observe that $$\Re\langle Av, Tv \rangle = \langle c_1 Tx_{\alpha_2} -c_2 Tx_{\alpha_1}, c_1 Tx_{\alpha_1} + c_2 Tx_{\alpha_2} \rangle = 0.$$ Thus we have $A \perp_\rho T.$ On the other hand, note that for any $z= \sum c_{\alpha_i} x_{\alpha_i}\in S_\mathbb{H}$ with $\sum |c_{\alpha_i}|^2 =1$, we get $\Re\langle Tz, Az\rangle = \frac{1}{2}\sum_{i\neq 1, 2}|c_{\alpha_i}|^2 \geq 0.$ Moreover,  it can be readily seen that there exists a $z \in M_T$ such that $\Re\langle Tz, Az\rangle >0,$ we obtain $T \not \perp_\rho A.$ This proves that $T$ is not $\rho$-right symmetric. 
\end{proof}	

In the next result we show a large class of operators in $\Li(\mathbb{H})$ cannot be $\rho$-right symmetric.
\begin{theorem}\label{th; right}
	Let $T\in \Li(\mathbb{H})$ be such that $M_T=S_{H_0}$ for some subspace $H_0$ of $\mathbb{H}.$ Suppose that $H_0$ has either of the following properties:
	\begin{itemize}
		\item[(i)] $codim (H_0)\geq 2$;
		\item[(ii)] $codim (H_0)=1$ and for $w_0\perp H_0,$ $Tw_0=0.$
	\end{itemize}
	Then $T$ is $\rho$-right symmetric if and only if $T=0$.
\end{theorem}
\begin{proof}
	Suppose that (i) holds true. Let us consider $\mathcal B=\{x_\alpha: \alpha\in \Lambda\}$ is an orthonormal basis for $H_0.$ Extend this basis to  $\mathcal{B}\cup\{x_\beta : \beta \in \Lambda'\}$ to get an orthonormal basis for $\mathbb{H}.$ Since $codim(H_0)\geq 2$, it follows that $codim(T(H_0))\geq 2.$ Therefore, there exist $w_0\in  S_\mathbb{H}\cap T(H_0)^{\perp}$ such that $w_0\perp Tx_\beta$ for some $\beta \in \Lambda'.$ Let us consider the following operator $A\in \Li(\mathbb{H})$:
	\begin{eqnarray*}
		Ax&=&\frac{1}{2} Tx \quad \forall\, x\in H_0,\\
		Ax_\beta&=& w_0,\\
		Ax_\gamma&=&0 \quad \forall\, \gamma\in \Lambda'\setminus \{\beta\}.
	\end{eqnarray*} 
	For any $z\in S_\mathbb{H},$ we can write $z= \sum c_{\alpha_i}x_{\alpha_i}+\sum c_{\beta_i}x_{\beta_i}$. Then one can see that $\|Az\|^2= |c_1|^2+ \frac{1}{4}\sum\|Tx_{\alpha_i}\|^2\leq 1.$ Therefore, $\|A\|=1$ and $M_A=\{\mu x_\beta: |\mu|=1\}.$ Now as $\langle Ax_\beta, Tx_\beta\rangle = 0,$ we have $A\perp_\rho T.$ On the other hand, for all $x\in M_T,$ we get $\langle Tx, Ax\rangle =\frac{1}{2},$ which implies $T\not\perp_\rho A.$\\
	For the condition (ii), we define $Ay = w_0$ and $Ah=\frac{1}{2}Th,$ where $h\in H_0$, $y\perp H_0$ and $w_0 \perp T(H_0).$  Now following similar argument as in (i), we can show that $T$ is not $\rho$-right symmetric.
\end{proof}
We now completely solve the problem of $\rho$-right symmetry in finite-dimensional case. To do so we need the following lemma.
\begin{lemma}\label{lem; diagonal}
	Let $\mathbb{H}$ be an $n$-dimensional Hilbert space with $\dim \mathbb{H}\geq 3$ and let $D\in \Li(\mathbb{H})$ be  diagonal. Then $D$ is $\rho$-right symmetric if and only if $D=0.$
\end{lemma}
\begin{proof}
	Let $D$ be diagonal with respect to the standard ordered basis $\{e_1, e_2, \ldots, e_n\}.$ 
	If $M_D=S_\mathbb{H}$ then from Theorem \ref{right; isometry} we obtain that $D$ is $\rho$-right symmetric if and only if $D=0.$ So, let us assume that $M_D= S_{H_0}\subsetneq S_\mathbb{H}.$ Now if $codim (H_0)\geq 2$ then the result follows from Theorem \ref{th; right} (i). Suppose that $codim(H_0)=1$ and without loss of generality assume that $M_D=S_{H_0}=span\{e_1, e_2, \ldots, e_{n-1}\}\cap S_{\mathbb{H}}$. If $De_n=0,$  again from Theorem \ref{th; right} (ii) we are done. So only case is left to show that when $De_n\neq 0.$    Let $De_i=\lambda_i e_i,$ for $1\leq i\leq n.$ Clearly, $|\lambda_i|=1,$ for all $1\leq i\leq n-1$ and $0<|\lambda_n|<1.$ Note that there are two possible cases for the values of $\lambda_1$ and $\lambda_n.$ \\
	\textbf{Case-I:} Let $\Re (\lambda_1\overline{\lambda_n})\neq 0.$ We consider an operator $A\in \Li(\mathbb{H})$ defined as $$Az=\bigg\langle z, \frac{1}{\sqrt{2}}(e_1+e_n) \bigg\rangle (\lambda_n e_1-\lambda_1e_n). $$  It is clear to observe that $\|A\|=\|\lambda_n e_1-\lambda_1e_n\|$ and $M_A=\{\pm \frac{1}{\sqrt{2}}(e_1+e_n)\}.$ Note that $$\Re\bigg\langle A\frac{1}{\sqrt{2}}(e_1+e_n), D\frac{1}{\sqrt{2}}(e_1+e_n)\bigg\rangle = 0.$$ This implies $A\perp_\rho D.$ On the other hand, for any $z\in M_D,$ we can write $z=\sum_{i=1}^{n-1} \alpha_ie_i.$ In this case, $Az= \frac{1}{\sqrt{2}}\alpha_1(\lambda_n e_1-\lambda_1e_n)$ and therefore, $$\Re\langle Dz, Az\rangle = \Re \frac{{\alpha_1}^2}{\sqrt{2}}\lambda_1\overline{\lambda_n}\neq 0, \,\, \forall z\in M_D. $$ So, either $\Re\langle Dz, Az\rangle> 0$ or, $\Re\langle Dz, Az\rangle< 0$ for all $z\in M_D$ This shows that $D\not\perp_\rho A.$ \\
	\textbf{Case-II:} Let $\Re(\lambda_1\overline{\lambda_n})=0.$ This implies either (i)	$\Re \lambda_1=0$  and $\Im \lambda_n=0$ or,
(ii) $\Im \lambda_1=0 $ and $\Re \lambda_n=0.$ Without loss of generality assume that (i) holds true. Then we consider $A\in \Li(\mathbb{H})$ as 
	$$Az=\bigg\langle z, \frac{-i}{\sqrt{2}}e_1+ \frac{1}{\sqrt{2}}e_n \bigg\rangle \bigg(\frac{\overline{\lambda_n} e_1}{\sqrt{2}}+\frac{i\overline{\lambda_1}e_n}{\sqrt{2}}\bigg).$$ Note that  $\|A\|=\|\frac{\overline{\lambda_n} e_1}{\sqrt{2}}+\frac{i\overline{\lambda_1}e_n}{\sqrt{2}}\|$ and $M_A=\{\mu (\frac{-i}{\sqrt{2}}e_1+ \frac{1}{\sqrt{2}}e_n): |\mu|=1 \}.$ Now proceeding similarly as in Case-I, we can arrive at the conclusion that $A\perp_\rho D$ whereas $D\not\perp_\rho A.$	Hence we obtain $D$ must be the zero operator, which completes the proof.
	\end{proof}
Now we prove our desired result.
\begin{theorem}\label{right; main}
	Let $\dim \mathbb{H}=n\geq 3$ and let $T\in \Li(\mathbb{H}).$ Then $T$ is $\rho$-right symmetric if and only if $T=0.$
\end{theorem}

\begin{proof}
  First of all let $T$ be a nonzero positive operator in $\Li(\mathbb{H}).$ Then by spectral theorem $T$ is unitariliy diagonalizable. Let $T=U^*DU,$ for some nonzero diagonal operator $D\in \Li(\mathbb{H})$ and $U\in \Li(\mathbb{H})$ is unitary. From Lemma \ref{lem; diagonal} we note that  there exists $A\in \Li(\mathbb{H})$ such that $A\perp_\rho D$ but $D\not\perp_\rho A.$ This implies from Lemma \ref{lem} that $U^*AU\perp_\rho T$ but $T\not\perp_\rho U^*AU.$ Hence $T$ is not $\rho$-right symmetric. Now for any $S\in \Li(\mathbb{H})$ we can write by polar decomposition that $S=U|S|$ or $S^*=V|S^*|,$  where $|S|=(S^*S)^{1/2},$ and  $U$ and $V$ can chosen as unitary (see \cite[Lemma 3.1]{Turnsek}). It is also easy to verify that for a unitary $U$, $T\perp_\rho A$ iff $UT\perp_\rho UA.$ Therefore, $|S|$ being nonzero positive operator, there exists $B \in \Li(\mathbb{H})$ such that $B\perp_\rho |S|$ but $|S|\not\perp_\rho B.$ This implies $UB\perp_\rho S$ but $S\not\perp_\rho UB.$ Also, since $S \perp_\rho A \iff S^*\perp_\rho A^*,$ it follows that the case $S^*=V|S^*|$ can be done by similar argument as before. Thus we arrive at our desired result. 
\end{proof}

We end this article with the following remark:
\begin{remark}
	We note from Corollary \ref{cor; sym} that in case of two-dimensional real Hilbert space, the shape of the numerical range of an operator is determined by the shape of its maximal numerical range. But if we consider $ \dim\mathbb{H}\geq 3,$ then this concept fails to hold, in general.  This leads to the following question:  \emph{Classify all the operators in $\Li(\mathbb{H})$ such that $\Re W_0(A)$ is symmetric with respect to the origin would imply $\Re W(A)$ is also symmetric with respect to the origin}. Equivalently, we seek to classify the following set:
	$$S:=\{A\in \Li(\mathbb{H}): A\perp_\rho I \implies I\perp_\rho A\}. $$ 
	\noindent Thus, the general theory of $\rho$-symmetry developed in this work not only extends classical results but also provides a unified and conceptually transparent framework for addressing this fundamental question concerning the geometric stability of numerical ranges in operator theory.
\end{remark}

\vspace{1cm}

\subsection*{Declarations}

\begin{itemize}
	
	\item Conflict of interest
	
	The authors have no relevant financial or non-financial interests to disclose.
	
	\item Data availability 
	
	The manuscript has no associated data.
	
	\item Author contribution
	
	All authors contributed to the study. All authors read and approved the final version of the manuscript.
	
\end{itemize}

\end{document}